\documentclass[a4paper,11pt]{article}
\usepackage{mathrsfs, amsmath, amsxtra, amssymb, latexsym, amscd, amsthm}
\newtheorem{thm}{Theorem}[section]

\newtheorem{lem}{Lemma}[section]
\newtheorem{prop}{Proposition}[section]
\theoremstyle{definition}
\newtheorem{defn}{Definition}[section]

\theoremstyle{remark}
\newtheorem{rem}{Remark}[section]
\numberwithin{equation}{section}
\def\ind{{\rm 1\hspace{-0.90ex}1}}
\begin{document}

\title{Explicit rates of convergence in the multivariate CLT for nonlinear statistics}
\author{Nguyen Tien Dung\footnote{Email: dung\_nguyentien10@yahoo.com}}

\date{\today}
\maketitle
\begin{abstract}
We investigate the multivariate central limit theorem for nonlinear statistics by means of Stein's method and Slepian's smart path interpolation method. Based on certain difference operators in theory of concentration inequalities, we obtain two explicit bounds for the rate of convergence. Applications to Rademacher functionals, the runs and quadratic forms are provided as well.
\end{abstract}
\noindent\emph{Keywords:} Multivariate normal approximation, Stein's method, Slepian's interpolation method, Difference operators.\\
{\em 2010 Mathematics Subject Classification:} 60F05, 62E17.

\section{Introduction}
Let $X=(X_1,X_2,...,X_n)$ be a vector of independent random variables (not necessarily identically distributed) taking values in some measurable space $\mathcal{X}.$ Fixed $d\geq 2,$ we consider the problem of obtaining explicit error bounds in the multivariate central limit theorem (CLT) for $\mathbb{R}^d$-valued random vector
\begin{equation}\label{11lqo}
F:=(F_1,F_2,...,F_d),
\end{equation}
where each $F_i:\mathcal{X}^n\to \mathbb{R}$ is a measurable function of $X,i.e.$ $F_i=F_i(X_1,X_2,...,X_n).$
The main task is to bound the distance
$$d_{\mathcal{H}}:=\sup\limits_{g\in \mathcal{H}}|E[g(F)]-E[g(Y)]|,$$
where $Y$ is a centered $d$-dimensional Gaussian vector and $\mathcal{H}$ is a suitable class of test functions. This problem, of course, is one of the most fundamental topics in Statistics and there is the number of works devoted to it. Among others, we refer the reader to Rinott \& Rotar \cite{Rinott1996} and  Chen \& Fang \cite{Chen2015} for the structures with local dependence, Bentkus \cite{Bentkus2003} and Chernozhukov et al. \cite{Chernozhukov2017} for studies of the dependence on dimension, Nourdin et al. \cite{Nourdin2010b} for homogeneous sums, D\"obler \& Peccati \cite{Dobler2017b} for $U$-statistics, etc. We also mention the general techniques such as the techniques of Malliavin calculus developed for the vectors of Gaussian, Poisson and Rademacher functionals \cite{Krokowski2016,Nourdin2010n,Peccati2010} and the technique of Stein couplings (exchangeable pairs, size bias couplings, etc) developed for arbitrary random vectors \cite{Chatterjee2008b,Fang2015,Goldstein1996,Reinert2009}.

It is surprising that only few works are devoted directly to the general random vectors (\ref{11lqo}). We only find in the literature two papers \cite{Bolthausen1993,Gotze1991} where Bolthausen and G\"otze  used Stein's method and linear statistics to establish Berry-Esseen bounds. Unfortunately, Theorem 2 in \cite{Bolthausen1993} is incorrect and a counterexample was given by Chen \& Shao, see Example 4.1 in \cite{Chen2007}. Regrading the technique and the results obtained in \cite{Gotze1991}, an exposition was given by Bhattacharya \& Holmes \cite{Bhattacharya2010}. In this paper, we do not aim to improve or generalize the results established previously by the other authors. Our purpose is to use a new technique for investigating the rate of convergence in the multivariate CLT for (\ref{11lqo}).

To measure the rate of convergence, we will provide the explicit upper estimates for the quantity
$$|E[g(F)]-E[g(Y)]|,$$
where the test function $g$ belongs to either $\mathcal{C}^2(\mathbb{R}^d)$ or $\mathcal{C}^3(\mathbb{R}^d).$ Those two classes of test functions were used in, e.g. \cite{Chatterjee2008b,Reinert2009,Peccati2010}. We recall that if $(F_n)_{n\geq 1}$ be a sequence of square integrable and centered random vectors and $ |E[g(F_n)]-E[g(Y)]|\to 0$ for any $g\in \mathcal{C}^k(\mathbb{R}^d)$ with bounded derivatives (for some $k\geq 1$), then $(F_n)_{n\geq 1}$ converges to $Y$ in distribution as $n$ tends to infinity. The steps in our proofs can be briefly described as follows.

\noindent{\it Step 1.} Using Stein's method and Slepian's interpolation method to reduce the problem to the study of covariances:
$$Cov(F_j,f_g(F)),\,\,1\leq j\leq d.$$
There is a common way to do this step, see e.g. \cite{Chernozhukov2017,Krokowski2016,Peccati2010}. In fact, the function $f_g:\mathbb{R}^d\to \mathbb{R}$ depends only on $g$ and is a twice differentiable function with bounded derivatives. %The reason we use the smooth test functions is due to this step.

\noindent{\it Step 2 (Main step).} Looking for the random variables $Z_{ij}$ satisfying
$$Cov(F_j,f_g(F))=E\big[\sum\limits_{i=1}^n\frac{\partial}{\partial x_i}f_g(F)Z_{ij}\big]+"remainder",\,\,1\leq j\leq d.$$
We note that if all $F_j's$ are Gaussian random variables, Stein's identity implies that $Z_{ij}=Cov(F_i,F_j)$ and $"remainder"$ vanishes.  Since $F_j's$ under our investigation are the measurable functions of independent random variables, we need a new technique to construct $Z_{ij}$ and to estimate $"remainder".$ %Thus the main contributions of the present paper lie in this step. Furthermore, when all $F_j's$ are living in the space of Gaussian, Poisson or Rademacher functionals, one can use the techniques of Malliavin calculus to construct $Z_{ij}$ as in \cite{Krokowski2016,Nourdin2010n,Peccati2010}.

\noindent{\it Step 3.} Combining the computations to get the explicit bounds for $|E[g(F)]-E[g(Y)]|.$

The rest of this article is organized as follows. Section \ref{vml4d9} contains the main ingradients in our work, we use the difference operators to construct a covariance formula and an approximate chain rule for the $\mathbb{R}$-valued functions of independent random variables. In Section \ref{mettt01}, we combine the results of Section \ref{vml4d9} with Stein's method and Slepian's interpolation method to obtain the explicit error bounds in multivariate CLT for the vectors (\ref{11lqo}). Some examples with detailed computations are given in Section \ref{mettt02}.

\section{Difference operators}\label{vml4d9}
Let $\mathcal{X}$ be a measurable space and $X=(X_1,X_2,...,X_n)$ be a vector of independent random variables, defined on some probability space $(\Omega,\mathfrak{F},P)$ and taking values in $\mathcal{X}.$ Let $X'=(X'_1,X'_2,...,X'_n)$ be an independent copy of $X.$ For each random variable $U=U(X),$ we write $T_iU=U(X_1,...,X_{i-1},X'_i,X_{i+1},...,X_n),i=1,...,n$ and denote by $E_i,E'_i$ the expectations with respect to $X_i$ and $X'_i,$ respectively. We introduce the $\sigma$-fields
\begin{align*}
&\mathcal{F}_0:=\{\emptyset,\Omega\}\,\,\,\text{and}\,\,\,\mathcal{F}_i:=\sigma(X_k,k\leq i),\,\,i=1,...,n
%&\text{$\mathcal{F}_\infty$ is the minimal $\sigma$-field generated by $(\mathcal{F}_i)_{i\geq 1}$}.
\end{align*}
and
\begin{align*}
&\mathcal{G}_{n+1}:=\{\emptyset,\Omega\}\,\,\,\text{and}\,\,\,\mathcal{G}_i:=\sigma(X_k,k\geq i),\,\,i=1,...,n.
%&\text{$\mathcal{F}_\infty$ is the minimal $\sigma$-field generated by $(\mathcal{F}_i)_{i\geq 1}$}.
\end{align*}
Following the notations introduced in \cite{Bobkov2017}, we recall the definition of two certain difference operators which will be used in our work.
%We first recall the definition of two certain difference operators in theory of concentration inequalities.
\begin{defn}\label{kod2ol} Given a random variable $U\in L^1(P),$ we define the difference operators $\mathfrak{D}_i$ by
$$\mathfrak{D}_iU=U-E_i[U],\,\,i=1,...,n.$$
When $U\in L^2(P),$ we define the difference operators $\mathfrak{d}_i$ by
$$\mathfrak{d}_iU=\big(\frac{1}{2}E'_i|U-T_iU|^2\big)^{\frac{1}{2}},\,\,i=1,...,n.$$
\end{defn}
We note that, in theory of Boolean functions, $\mathfrak{D}$ is the so-called Laplacian operator, see e.g. Definition 2.25 in \cite{Donnell2014}.  The operators $\mathfrak{D}$ and $\mathfrak{d}$ both are very useful in the study of concentration inequalities.  In particular, the Efron-Stein inequality formulated in Theorem 3.1 of \cite{Boucheron2013} can be restated as follows.
\begin{prop} (Efron-Stein inequality) For any random variable $U\in L^2(P),$ we have
\begin{equation}\label{sopq1}
Var(U)\leq \sum\limits_{i=1}^n E|\mathfrak{D}_iU|^2=\sum\limits_{i=1}^n E|\mathfrak{d}_iU|^2.
\end{equation}
\end{prop}
Let us now recall some useful properties of the operators $\mathfrak{D}$ and $\mathfrak{d}$, see e.g. \cite{Bobkov2017,Dungnt2018}. For the sake of completeness we will give a brief proof of those properties.
\begin{prop}\label{mcsk4} For each $i= 1,...,n,$ under suitable integrability assumptions, we have

%(i) $E[\mathfrak{D}_iF]=0,$ %and $E[\mathfrak{D}_iF]=0,$

%(ii) $\mathfrak{D}_iE[F|\mathcal{F}_k]=0,\,\,\forall\,k<i,$

\noindent (i) $\mathfrak{D}_iE[U|\mathcal{F}_i]=E[\mathfrak{D}_iU|\mathcal{F}_i]$ and $\mathfrak{D}_iE[U|\mathcal{G}_i]=E[\mathfrak{D}_iU|\mathcal{G}_i],$

\noindent (ii) $E\left[(\mathfrak{D}_iU)V\right]=E\left[(\mathfrak{D}_iV)U\right]=E\left[(\mathfrak{D}_iU)(\mathfrak{D}_iV)\right],$

\noindent (ii) $(\mathfrak{d}_iU)^2=\frac{1}{2}[(\mathfrak{D}_iU)^2+E_i(\mathfrak{D}_iU)^2],$

%(vi) $\mathfrak{D}_i(FG)=F\mathfrak{D}_iG+G\mathfrak{D}_iF-\mathfrak{D}_iF\mathfrak{D}_iG-E_i[\mathfrak{D}_iF\mathfrak{D}_iG],$

\noindent (iv) $E|\mathfrak{D}_iU|^p\leq 2^p E|U|^p$ and $E[(\mathfrak{d}_iU)^{2p}]\leq E[(\mathfrak{D}_iU)^{2p}]$ $\,\,\forall\,p\geq 1.$ Particularly, $E|\mathfrak{D}_iU|^2\leq E|U|^2.$

%(vi) $(E[\mathfrak{D}_iF|\mathcal{F}_i])_{i\geq 1}$ is an orthogonal sequence of centered random variables, that is, $E[E[\mathfrak{D}_iF|\mathcal{F}_i]E[\mathfrak{D}_kF|\mathcal{F}_k]]=0$ for $i\neq k.$
\end{prop}
\begin{proof}
$(i)$ By the independence, we have $E_i[U]=E[U|\sigma(X_k,k\neq i)].$ Hence, we obtain
$$\mathfrak{D}_iE[U|\mathcal{F}_i]=E[U|\mathcal{F}_i]-E[E[U|\mathcal{F}_i]|\sigma(X_k,k\neq i)]=E[U|\mathcal{F}_i]-E[U|\mathcal{F}_{i-1}]$$
and
$$E[\mathfrak{D}_iU|\mathcal{F}_i]=E[U|\mathcal{F}_i]-E[E[U|\sigma(X_k,k\neq i)]|\mathcal{F}_i]=E[U|\mathcal{F}_i]-E[U|\mathcal{F}_{i-1}].$$
Similarly, we also obtain $\mathfrak{D}_iE[U|\mathcal{G}_i]=E[U|\mathcal{G}_{i}]-E[U|\mathcal{G}_{i+1}]=E[\mathfrak{D}_iU|\mathcal{G}_i].$

\noindent$(ii)$ This point follows from the relation
$$E[E_i[U]E_i[V]]=E[UE_i[V]]=E[E_i[U]V].$$
%We always have, see Lemma 5.1 in \cite{Gotze2017}
\noindent$(iii)$ Because $E_i[U]=E'_i[T_iU],$ we have
\begin{align*}
E'_i[(E_iU-T_iU)^2]&=E_i[(E_iU-U)^2]\\
&=E_i(\mathfrak{D}_iU)^2.
\end{align*}
This, together with the decomposition
$(U-T_iU)^2=(U-E_iU)^2+2(U-E_iU)(E_iU-T_iU)+(E_iU-T_iU)^2,$ gives us
$$2(\mathfrak{d}_iU)^2=(\mathfrak{D}_iU)^2+E'_i[(E_iU-T_iU)^2]=(\mathfrak{D}_iU)^2+E_i(\mathfrak{D}_iU)^2.$$
So we can finish the proof of the point $(iii).$

\noindent$(iv)$ By using the fundamental inequality $(a+b)^p\leq 2^{p-1}(a^p+b^p)$ we obtain
$$E|\mathfrak{D}_iU|^p\leq 2^{p-1} (E|U|^p+E|E_i[U]|^p)\leq 2^{p}E|U|^p,\,\,p\geq 1.$$
Similarly,
\begin{align*}E[(\mathfrak{d}_iU)^{2p}]&=\frac{1}{2^p}E[((\mathfrak{D}_iU)^2+E_i(\mathfrak{D}_iU)^2)^p]\\
&\leq \frac{1}{2}E[(\mathfrak{D}_iU)^{2p}+E_i(\mathfrak{D}_iU)^{2p}]\\
&\leq E[(\mathfrak{D}_iU)^{2p}],\,\,p\geq 1.
\end{align*}
When $p=2,$ we have $E|\mathfrak{D}_iU|^2=E[U^2]-2E[UE_i[U]]+E[(E_i[U])^2]=E[U^2]-E[(E_i[U])^2]\leq E[U^2].$

The proof of Proposition is complete.
\end{proof}
%The key result of this paper is formulated in the following theorem where we obtain  a new covariance formula.
The next two propositions provide us the main ingradients to perform {\it Step 2} mentioned in Introduction.
\begin{prop}\label{lods3} (Covariance formula) Let $U=U(X)$ and $V=V(X)$ be two random variables in $L^2(P).$ For any $\alpha\in[0,1],$ we have
$$Cov(U,V)=E\left[\sum\limits_{i=1}^n \mathfrak{D}_i U \mathfrak{D}^{(\alpha)}_i V\right],$$%=E\left[\sum\limits_{i=1}^n D_i G E[D_iF|\mathcal{F}_i]\right]$$
where $\mathfrak{D}^{(\alpha)}_i V:=\alpha E[\mathfrak{D}_iV|\mathcal{F}_i]+(1-\alpha)E[\mathfrak{D}_iV|\mathcal{G}_i].$
\end{prop}
\begin{proof}We have
\begin{align*}
V-E[V]&=\sum\limits_{i=1}^n(E[V|\mathcal{F}_i]-E[V|\mathcal{F}_{i-1}])\\
&=\sum\limits_{i=1}^n \mathfrak{D}_iE[V|\mathcal{F}_i]\,\,\text{by Proposition \ref{mcsk4}, $(i)$}.
\end{align*}
Similarly, $V-E[V]=\sum\limits_{i=1}^n(E[V|\mathcal{G}_{i}]-E[V|\mathcal{G}_{i+1}])=\sum\limits_{i=1}^n \mathfrak{D}_iE[V|\mathcal{G}_i].$ Hence, we can get
\begin{align*}
&Cov(U,V)=E[U(V-E[V])]\\
&=\alpha\sum\limits_{i=1}^n E\left[U\mathfrak{D}_iE[V|\mathcal{F}_i]\right]+(1-\alpha)\sum\limits_{i=1}^n E\left[U\mathfrak{D}_iE[V|\mathcal{G}_i]\right]\\
&=\alpha\sum\limits_{i=1}^n E\left[\mathfrak{D}_iU\mathfrak{D}_iE[V|\mathcal{F}_i]\right]+(1-\alpha)\sum\limits_{i=1}^n E\left[\mathfrak{D}_iU\mathfrak{D}_iE[V|\mathcal{G}_i]\right]\,\,\text{by Proposition \ref{mcsk4}, $(ii)$}\\
&=\alpha\sum\limits_{i=1}^n E\left[\mathfrak{D}_iUE[\mathfrak{D}_iV|\mathcal{F}_i]\right]+(1-\alpha)\sum\limits_{i=1}^n E\left[\mathfrak{D}_iUE[\mathfrak{D}_iV|\mathcal{G}_i]\right]\,\,\text{by Proposition \ref{mcsk4}, $(i)$}.
\end{align*}
This completes the proof.
\end{proof}
\begin{prop}\label{lods33}(Approximate chain rule) Consider a random vector $F=(F_1,...,F_d)\in L^2(P),$ where each component is a measurable function of $X.$ For any the function $f\in \mathcal{C}^2(\mathbb{R}^d)$ with bounded derivatives, we have
\begin{align*}
\mathfrak{D}_kf(F)&=\sum\limits_{i=1}^d \frac{\partial}{\partial x_i}f(F) \mathfrak{D}_kF_i+\sum\limits_{i,j=1}^d R^{(k,f)}_{ij},\,\,1\leq k\leq n,
\end{align*}
where the remainder terms $R^{(k,f)}_{ij},1\leq i,j\leq d$ satisfy the bound
$$|R^{(k,f)}_{ij}|\leq \frac{1}{2}\sup\limits_{x\in \mathbb{R}^d}\big|\frac{\partial^2}{\partial x_i\partial x_j}f(x)\big|[(\mathfrak{d}_kF_i)^2+(\mathfrak{d}_kF_j)^2].$$
\end{prop}
\begin{proof} By the multivariate Taylor expansion we have
$$f(x)-f(y)=\sum\limits_{i=1}^d \frac{\partial}{\partial x_i}f(x) (x_i-y_i)+\sum\limits_{i,j=1}^d R^{(f)}_{ij}$$
for all $x,y\in \mathbb{R}^d,$ where the remainder terms $R^{(f)}_{ij}$ are bounded by
$$|R^{(f)}_{ij}|\leq \frac{1}{2}\sup\limits_{x\in \mathbb{R}^d}\big|\frac{\partial^2}{\partial x_i\partial x_j}f(x)\big|\times |(x_i-y_i)(x_j-y_j)|.$$
On the other hand, for each $k=1,...,n,$ we have
\begin{align*}
\mathfrak{D}_kf(F)=f(F)-E_k[f(F)]=f(F)-E'_k[f(T_kF)]=E'_k[f(F)-f(T_kF)],
\end{align*}
where $T_kF=(T_kF_1,T_kF_2,...,T_kF_d).$ Hence, we can write
\begin{align*}
\mathfrak{D}_kf(F)&=E'_k\big[\sum\limits_{i=1}^d \frac{\partial}{\partial x_i}f(F) (F_i-T_kF_i)\big]+\sum\limits_{i,j=1}^d R^{(k,f)}_{ij}\\
&=\sum\limits_{i=1}^d \frac{\partial}{\partial x_i}f(F) \mathfrak{D}_kF_i+\sum\limits_{i,j=1}^d R^{(k,f)}_{ij},\,\,k=1,...,n,
\end{align*}
where the remainder terms $R^{(k,f)}_{ij}$ are bounded by
\begin{align*}
|R^{(k,f)}_{ij}|&\leq \frac{1}{2}\sup\limits_{x\in \mathbb{R}^d}\big|\frac{\partial^2}{\partial x_i\partial x_j}f(x)\big|E'_k|(F_i-T_kF_i)(F_j-T_kF_j)|\\
&\leq\frac{1}{4}\sup\limits_{x\in \mathbb{R}^d}\big|\frac{\partial^2}{\partial x_i\partial x_j}f(x)\big|(E'_k|F_i-T_kF_i|^2+E'_k|F_j-T_kF_j|^2)\\
&=\frac{1}{2}\sup\limits_{x\in \mathbb{R}^d}\big|\frac{\partial^2}{\partial x_i\partial x_j}f(x)\big|[(\mathfrak{d}_kF_i)^2+(\mathfrak{d}_kF_j)^2].
\end{align*}
The proof is complete.
\end{proof}

\section{Explicit rates of convergence}\label{mettt01} %To formulate the results we need more notation repeating most part of the notation used in
In this Section, we employ Stein's method and Slepian's interpolation method to obtain two explicit bounds for rates of convergence. To begin, we recall some basic notations.

\noindent $\bullet$ On the space of real $d\times d$ matrices, the Hilbert-Schmidt inner product and the Hilbert - Schmidt norm are defined respectively by $\langle A,B\rangle_{H.S.}:=\mathrm{Tr}(AB^T)$ and $\|A\|_{H.S.}:=\sqrt{\mathrm{Tr}(AA^T)}.$ The operator norm of a matrix $A$ is defined by $\|A\|_{op}:=\sup\limits_{\|x\|_{\mathbb{R}^d}=1}\|Ax\|_{\mathbb{R}^d},$ where $\|.\|_{\mathbb{R}^d}$ is the Euclidian norm on $\mathbb{R}^d.$ Note that if $A=diag(\lambda_1,...,\lambda_d)$ is a diagonal matrix, then $\|A\|_{op}=\max\limits_{1\leq i\leq d}|\lambda_i|.$

\noindent $\bullet$ $\mathcal{C}^k(\mathbb{R}^d)$ denotes the space of $k$-times continuously differentiable real-valued functions on $\mathbb{R}^d.$

\noindent $\bullet$ For every function $g:\mathbb{R}^d\to \mathbb{R},$ let
$$\|g\|_{Lip}:=\sup\limits_{x\neq y}\frac{|g(x)-g(y)|}{\|x-y\|_{\mathbb{R}^d}}.$$
We also let
$$\text{$M_2(g):=\sup\limits_{x\neq y}\frac{\|\nabla g(x)-\nabla g(y)\|_{\mathbb{R}^d}}{\|x-y\|_{\mathbb{R}^d}}$ if $g\in \mathcal{C}^1(\mathbb{R}^d)$}$$
and
$$\text{$M_3(g):=\sup\limits_{x\neq y}\frac{\|\mathrm{Hess}\,g(x)-\mathrm{Hess}\,g(y)\|_{op}}{\|x-y\|_{\mathbb{R}^d}}$ if $g\in \mathcal{C}^2(\mathbb{R}^d)$}.$$
Note that if $g\in \mathcal{C}^1(\mathbb{R}^d),$ then $\|g\|_{Lip}=\sup\limits_{x\in \mathbb{R}^d}\|\nabla g(x)\|_{\mathbb{R}^d}.$ If $g\in \mathcal{C}^2(\mathbb{R}^d),$ then $M_2(g)=\sup\limits_{x\in \mathbb{R}^d}\|\mathrm{Hess}\,g(x)\|_{op}.$

\noindent $\bullet$ For a positive integer $k$ and a function $g\in \mathcal{C}^k(\mathbb{R}^d),$ we put
$$\|g^{(k)}\|_\infty:=\max\limits_{1\leq i_1\leq...\leq i_k\leq d}\sup\limits_{x\in \mathbb{R}^d}\big|\frac{\partial^k}{\partial x_{i_1}...\partial x_{i_k}}g(x)\big|.$$
As usual, we write $g^{(2)}=g''$ and $g^{(3)}=g'''.$

\subsection{Stein's method}
Powerful as it is, Stein's method has been extensively used to study the rate of convergence in CLTs. In multivariate setting, some elements of this method can be summarized as in the next lemma (see, e.g. Lemma 2.17 in \cite{Peccati2010}).
\begin{lem} Fix an integer $d\geq  2$ and let $C = \{C_{ij} : i, j =1, . . . , d\}$ be a $d \times d$ nonnegative definite symmetric real matrix.

\noindent 1. Let $Y$ be a random variable with values in $\mathbb{R}^d.$ Then $Y \sim N_d (0, C)$ if and only if, for every twice differentiable function $f:\mathbb{R}^d\to \mathbb{R}$ such that $E|\langle C,\mathrm{Hess} f(Y)\rangle_{H.S.}|+E|\langle Y,\nabla f(Y)\rangle_{\mathbb{R}^d}|<\infty,$ it holds that
$$E[\langle Y,\nabla f(Y)\rangle_{\mathbb{R}^d}-\langle C,\mathrm{Hess} f(Y)\rangle_{H.S.}]=0.$$

\noindent 2. Assume in addition that $C$ is positive definite and consider a Gaussian random vector $Y \sim N_d (0, C).$ Let $g:\mathbb{R}^d\to \mathbb{R}$ belong to $\mathcal{C}^2(\mathbb{R}^d)$ with first and second bounded derivatives. Then, the function $U_0g$ defined by
\begin{equation}\label{mull01}
U_0g(x)=\int_0^1 \frac{1}{2t}E[g(\sqrt{t}x+\sqrt{1-t}Y)-g(Y)]dt
\end{equation}
is a solution to the following partial differential equation (with unknown function $f $):
\begin{equation}\label{mull02}
g(x)-E[g(Y)]=\langle x,\nabla f(x)\rangle_{\mathbb{R}^d}-\langle C,\mathrm{Hess} f(x)\rangle_{H.S.},\,\,\,x\in \mathbb{R}^d.
\end{equation}
Moreover, one has that
\begin{equation}\label{uudk1}
\sup\limits_{x\in \mathbb{R}^d}\|\mathrm{Hess}\, U_0g(x)\|_{H.S.}\leq \|C^{-1}\|_{op}\|C\|^{1/2}_{op}\|g\|_{Lip}
\end{equation}
and
\begin{equation}\label{uudk2}
M_3(U_0g)\leq \frac{\sqrt{2\pi}}{4} \|C^{-1}\|^{3/2}_{op}\|C\|_{op}M_2(g)
\end{equation}
\end{lem}

The next statement is the first main result of the present paper.
\begin{thm}\label{k1opa4} Let $C=(C_{ij})_{d\times d}$ be a positive definite matrix and $Y$ be a centered $d$-dimensional Gaussian vector with covariance $C.$ Suppose that $F=(F_1,...,F_d)$ is a $\mathbb{R}^d$-valued random vector such that $E[F_i]=0$ and $\sigma_{ij}:=E[F_iF_j]<\infty$ for all $0\leq i,j\leq d.$ Then, for any $\alpha\in[0,1]$ and $g\in \mathcal{C}^2(\mathbb{R}^d)$ with $\|g\|_{Lip}+M_2(g)<\infty,$ we have
\begin{align}
|E[g(F)]&-E[g(Y)]|\leq B_1\big(\sum\limits_{i,j=1}^dE|C_{ij}-Z^{(\alpha)}_{ij}|^2\big)^{1/2}+B_2\sum\limits_{i=1}^d\sum\limits_{k=1}^nE|\mathfrak{D}_k F_i|^3,\label{jjflme}\\
&\leq B_1\sum\limits_{i,j=1}^d|C_{ij}-\sigma_{ij}|+B_1\sum\limits_{i,j=1}^d \sqrt{Var(Z^{(\alpha)}_{ij})}+B_2\sum\limits_{i=1}^d \sum\limits_{k=1}^nE|\mathfrak{D}_k F_i|^3,\label{9hlw3g}
%&\leq B_1\big(\sum\limits_{i,j=1}^d|C_{ij}-\sigma_{ij}|^2\big)^{1/2}+B_1\big(\sum\limits_{i,j=1}^d Var(Z^{(\alpha)}_{ij})\big)^{1/2}+B_2\sum\limits_{i=1}^d \sum\limits_{k=1}^nE|\mathfrak{D}_k F_i|^3
\end{align}
where $B_1:=\|C^{-1}\|_{op}\|C\|^{1/2}_{op} \|g\|_{Lip},B_2:=\frac{\sqrt{2\pi}}{4} \|C^{-1}\|^{3/2}_{op}\|C\|_{op}M_2(g)d^2$ and $$Z^{(\alpha)}_{ij}:=\sum\limits_{k=1}^n\mathfrak{D}_k F_i \mathfrak{D}^{(\alpha)}_k F_j,\,1\leq i,j\leq d.$$
\end{thm}
\begin{proof}{\it Step 1.} By using an approximate argument as in the proof of Theorem 2.3 in \cite{Chatterjee2008b}, we can and will assume that $g\in \mathcal{C}^\infty(\mathbb{R}^d).$ Because the function $U_0g(x)$ defined by (\ref{mull01}) is a solution to the equation (\ref{mull02}) we obtain
\begin{align}
|E[g(F)]-E[g(Y)]|&=|E[\langle C,\mathrm{Hess}\, U_0g(F)\rangle_{H.S.}-\langle F,\nabla U_0g(F)\rangle_{\mathbb{R}^d}]|\notag\\
&=\big|E\big[\sum\limits_{i,j=1}^dC_{ij}\frac{\partial^2}{\partial x_i\partial x_j}U_0g(F)-\sum\limits_{j=1}^dF_j\frac{\partial}{\partial x_j}U_0g(F)\big]\big|\notag\\
&=\big|E\big[\sum\limits_{i,j=1}^dC_{ij}\frac{\partial^2}{\partial x_i\partial x_j}U_0g(F)-\sum\limits_{j=1}^dF_jf_j(F)\big]\big|,\label{iiflq}
\end{align}
where $f_j(x):=\frac{\partial}{\partial x_j}U_0g(x),\,\,j=1,...,d,\,\,x\in \mathbb{R}^d.$ Note that
$$\frac{\partial}{\partial x_i}f_j(x)=\frac{\partial^2}{\partial x_i \partial x_j}U_0g(x),\,\,\,\frac{\partial^2}{\partial x_i \partial x_l}f_j(x)=\frac{\partial^3}{\partial x_i \partial x_l \partial x_j}U_0g(x).$$
{\it Step 2.} For each $j=1,...,d,$ it follows from Proposition \ref{lods33} that
$$\mathfrak{D}_kf_j(F)=\sum\limits_{i=1}^d \frac{\partial^2}{\partial x_i \partial x_j}U_0g(F) \mathfrak{D}_kF_i+\sum\limits_{i,l=1}^d R^{(k,f_j)}_{il},\,\,1\leq k\leq n,$$
where the remainder terms $R^{(k,f_j)}_{il}$ are bounded by
\begin{equation}\label{kdlw1}
|R^{(k,f_j)}_{il}|\leq \frac{ \|f_j''\|_\infty}{2}[(\mathfrak{d}_kF_i)^2+(\mathfrak{d}_kF_l)^2]\leq \frac{ \|(U_0g)'''\|_\infty}{2}[(\mathfrak{d}_kF_i)^2+(\mathfrak{d}_kF_l)^2].
\end{equation}
We have $E\big[F_jf_j(F)\big]=Cov(F_j,f_j(F))$ because $E[F_j]=0.$ Thanks to Proposition \ref{lods3} we obtain
\begin{align}
E\big[F_jf_j(F)\big]&=E\big[\sum\limits_{k=1}^n \mathfrak{D}_kf_j(F) \mathfrak{D}^{(\alpha)}_kF_j\big]\notag\\
&=E\big[\sum\limits_{k=1}^n \big(\sum\limits_{i=1}^d \frac{\partial^2}{\partial x_i \partial x_j}U_0g(F) \mathfrak{D}_kF_i+\sum\limits_{i,l=1}^d R^{(k,f_j)}_{il}\big)\mathfrak{D}^{(\alpha)}_kF_j\big]\notag\\
&=E\big[\sum\limits_{i=1}^d \frac{\partial^2}{\partial x_i \partial x_j}U_0g(F) \sum\limits_{k=1}^n \mathfrak{D}_kF_i \mathfrak{D}^{(\alpha)}_kF_j+\sum\limits_{k=1}^n \big(\sum\limits_{i,l=1}^d R^{(k,f_j)}_{il}\big)\mathfrak{D}^{(\alpha)}_kF_j\big]\notag\\
&=E\big[\sum\limits_{i=1}^d \frac{\partial^2}{\partial x_i \partial x_j}U_0g(F) Z^{(\alpha)}_{ij}+\sum\limits_{k=1}^n \big(\sum\limits_{i,l=1}^d R^{(k,f_j)}_{il}\big)\mathfrak{D}^{(\alpha)}_kF_j\big],\,\,1\leq j\leq d.\notag
\end{align}
As a consequence, we can write
\begin{align}
E\big[\sum\limits_{j=1}^dF_jf_j(F)\big]=E\big[\sum\limits_{i,j=1}^d\frac{\partial^2}{\partial x_i\partial x_j}U_0g(F)Z^{(\alpha)}_{ij}\big]+R,\label{iiflq1}
\end{align}
where $R:=E\big[\sum\limits_{j=1}^d\sum\limits_{k=1}^n \big(\sum\limits_{i,l=1}^d R^{(k,f_j)}_{il}\big)\mathfrak{D}^{(\alpha)}_kF_j\big].$ From the estimates (\ref{kdlw1}) we deduce
\begin{align*}
|R|&\leq \frac{\|(U_0g)'''\|_\infty}{2}E\big[\sum\limits_{j=1}^d\sum\limits_{k=1}^n \big(\sum\limits_{i,l=1}^d[(\mathfrak{d}_kF_i)^2+(\mathfrak{d}_kF_l)^2]\big)|\mathfrak{D}^{(\alpha)}_kF_j|\big]\\
&=\|(U_0g)'''\|_\infty d\sum\limits_{k=1}^n E\big[\big(\sum\limits_{i=1}^d(\mathfrak{d}_kF_i)^2\big)\big(\sum\limits_{j=1}^d|\mathfrak{D}^{(\alpha)}_kF_j|\big)\big].
\end{align*}
By the elementary inequality $(|a_1|+...+|a_d|)^m\leq d^{m-1}(|a_1|^p+...+|a_d|^m)$ for all $a_1,...,a_d\in \mathbb{R}$ and $m\geq 1,$ we obtain
$$E\big(\sum\limits_{i=1}^d(\mathfrak{d}_kF_i)^2\big)^{3/2}\leq \sqrt{d}E\big[\sum\limits_{i=1}^d(\mathfrak{d}_kF_i)^3\big]\leq \sqrt{d}\sum\limits_{i=1}^d E|\mathfrak{D}_kF_i|^3,$$
$$E\big(\sum\limits_{j=1}^d |\mathfrak{D}^{(\alpha)}_kF_j|\big)^3\leq d^2E\big[\sum\limits_{j=1}^d|\mathfrak{D}^{(\alpha)}_kF_j|^3\big]\leq d^2\sum\limits_{j=1}^d E|\mathfrak{D}_kF_j|^3.$$
Note that, in the last inequality, we used the facts that $|\mathfrak{D}^{(\alpha)}_k F_j|^3\leq (\alpha |E[\mathfrak{D}_kF_j|\mathcal{F}_k]|+(1-\alpha)|E[\mathfrak{D}_kF_j|\mathcal{G}_k]|)^3$ and that $E[|E[\mathfrak{D}_kF_j|\mathcal{F}_k]|^p|E[\mathfrak{D}_kF_j|\mathcal{F}_k]|^{3-p}]\leq E|\mathfrak{D}_kF_j|^3$ for all $0\leq p\leq 3.$

We now use H\"older inequality and the relation $\|(U_0g)'''\|_\infty\leq M_3(U_0g)$ to get
\begin{align}
|R|&\leq \|(U_0g)'''\|_\infty d\sum\limits_{k=1}^n\big(\sqrt{d}\sum\limits_{i=1}^d E|\mathfrak{D}_kF_i|^3\big)^{2/3}\big(d^2\sum\limits_{i=1}^d E|\mathfrak{D}_kF_i|^3\big)^{1/3}\notag\\
&\leq  M_3(U_0g) d^2 \sum\limits_{i=1}^d \sum\limits_{k=1}^nE|\mathfrak{D}_kF_i|^3.\label{o2md1}
\end{align}
{\it Step 3.} Inserting (\ref{iiflq1})  into (\ref{iiflq}) yields
\begin{align}
|E[g(F)]-E[g(Y)]|&=\big|E\big[\sum\limits_{i,j=1}^d(C_{ij}-Z^{(\alpha)}_{ij})\frac{\partial^2}{\partial x_i\partial x_j}U_0g(F)-R\big|,\notag\\
&\leq \big|E\big[\sum\limits_{i,j=1}^d(C_{ij}-Z^{(\alpha)}_{ij})\frac{\partial^2}{\partial x_i\partial x_j}U_0g(F)\big|+|R|,\notag\\
&\leq \sqrt{E[\|\mathrm{Hess}\,U_0g(F)\|^2_{H.S.}]}\sqrt{\sum\limits_{i,j=1}^dE|C_{ij}-Z^{(\alpha)}_{ij}|^2}+|R|.\notag
\end{align}
So we can obtain (\ref{jjflme}) by using (\ref{uudk1}), (\ref{uudk2}) and (\ref{o2md1}). To finish the proof we observe from Proposition \ref{lods3} that $E[Z^{(\alpha)}_{ij}]=\sigma_{ij}.$ Hence,
$$E|C_{ij}-Z^{(\alpha)}_{ij}|^2=|C_{ij}-\sigma_{ij}|^2+Var(Z^{(\alpha)}_{ij}),\,\,1\leq i,j\leq d.$$
By the elementary inequality $\sqrt{a+b}\leq \sqrt{a}+\sqrt{b}$ for any $a,b\geq 0,$ we get
$$\big(\sum\limits_{i,j=1}^dE|C_{ij}-Z^{(\alpha)}_{ij}|^2\big)^{1/2}\leq \sum\limits_{i,j=1}^d |C_{ij}-\sigma_{ij}|+\sum\limits_{i,j=1}^d\sqrt{Var(Z^{(\alpha)}_{ij})}.$$
So we obtain (\ref{9hlw3g}) from (\ref{jjflme}).
\end{proof}
\subsection{Slepian's interpolation method}
We observe that Stein's method requires the positive definite property of covariance matrix $C.$ In addition, the operator norms of $C$ and $C^{-1}$ are not easy to compute in many practical problems. Slepian's interpolation method will help us to avoid these disadvantages. However, the price to pay is that we have to use the test functions in $\mathcal{C}^3(\mathbb{R}^d)$ instead of $\mathcal{C}^2(\mathbb{R}^d).$ The next theorem contains the second main result of the present paper.
\begin{thm}\label{ddlm3} Let $Y$ be a centered $d$-dimensional Gaussian vector with covariance matrix $C=(C_{ij})_{d\times d}$ (not
necessarily positive definite). Suppose that $F=(F_1,...,F_d)$ is a $\mathbb{R}^d$-valued random vector such that $E[F_i]=0$ and $\sigma_{ij}:=E[F_iF_j]<\infty$ for all $0\leq i,j\leq d.$ Then, for any  $\alpha\in[0,1]$ and $g\in \mathcal{C}^3(\mathbb{R}^d)$ with $\|g''\|_\infty+\|g'''\|_\infty<\infty,$ we have
\begin{align}
|E[g(F)]&-E[g(Y)]|\leq B_3\sum\limits_{i,j=1}^dE|C_{ij}-Z^{(\alpha)}_{ij}|+B_4 \sum\limits_{i=1}^d \sum\limits_{k=1}^nE|\mathfrak{D}_k F_i|^3\label{9hlf4dj}\\
&\leq B_3\sum\limits_{i,j=1}^d|C_{ij}-\sigma_{ij}|+B_3\sum\limits_{i,j=1}^d \sqrt{Var(Z^{(\alpha)}_{ij})}+B_4\sum\limits_{i=1}^d \sum\limits_{k=1}^nE|\mathfrak{D}_k F_i|^3,\label{9hlw3g3}
\end{align}
where $B_3:=\frac{\|g''\|_\infty}{2} ,B_4:=\frac{\|g'''\|_\infty d^2}{3}$ and $Z^{(\alpha)}_{ij},1\leq i,j\leq d$ are as in Theorem \ref{k1opa4}.

\end{thm}
\begin{proof}{\it Step 1.} Without loss of generality, we can assume that $F$ and $Y$ are independent. Consider the Slepian's interpolation function $H(t)$ defined by
$$H(t)=E[g(\sqrt{1-t}F+\sqrt{t}Y)],\,\,t\in [0,1].$$
Clearly, $H(t)$ is differentiable on $(0,1)$ and its derivative is given by
\begin{align}
&H'(t)=E\big[\sum\limits_{j=1}^d\frac{\partial}{\partial x_j}g(\sqrt{1-t}F+\sqrt{t}Y)\big(-\frac{F_j}{2\sqrt{1-t}}+\frac{Y_j}{2\sqrt{t}}\big)\big]\notag\\
&=\frac{E\big[\sum\limits_{j=1}^d\frac{\partial}{\partial x_j}g(\sqrt{1-t}F+\sqrt{t}Y)Y_j\big]}{2\sqrt{t}}-\frac{E\big[\sum\limits_{j=1}^d\frac{\partial}{\partial x_j}g(\sqrt{1-t}F+\sqrt{t}Y)F_j\big]}{2\sqrt{1-t}}.\label{k9}
\end{align}
{\it Step 2.} By using Stein's identity (see, e.g. Appendix A.6 in \cite{Talagrand2003}) we obtain
\begin{align}
E\big[\sum\limits_{j=1}^d\frac{\partial}{\partial x_j}g(\sqrt{1-t}F+\sqrt{t}Y)Y_j\big]&=E\bigg[E\big[\sum\limits_{j=1}^d\frac{\partial}{\partial x_j}g(\sqrt{1-t}a+\sqrt{t}Y)Y_j\big]|_{a=F}\bigg]\notag\\
&=\sqrt{t}E\big[\sum\limits_{i,j=1}^d\frac{\partial^2}{\partial x_i\partial x_j}g(\sqrt{1-t}F+\sqrt{t}Y)C_{ij}\big].\label{steinko2}
\end{align}
Fixed $t\in(0,1),b\in \mathbb{R}^d,$ we consider the functions $f^{t,b}_j(x):=\frac{\partial}{\partial x_j}g(\sqrt{1-t}x+\sqrt{t}b),\,\,j=1,...,d,\,\,x\in \mathbb{R}^d.$ Notice that
$$\frac{\partial}{\partial x_i}f^{t,b}_j(x)=\sqrt{1-t}\frac{\partial^2}{\partial x_i \partial x_j}g(\sqrt{1-t}x+\sqrt{t}b),$$
$$\frac{\partial^2}{\partial x_i \partial x_l}f^{t,b}_j(x)=(1-t)\frac{\partial^3}{\partial x_i \partial x_l \partial x_j}g(\sqrt{1-t}x+\sqrt{t}b).$$
For each $j=1,...,d,$ we apply Proposition \ref{lods33} to $f^{t,b}_j$ and we obtain
$$\mathfrak{D}_kf^{t,b}_j(F)=\sqrt{1-t}\sum\limits_{i=1}^d \frac{\partial^2}{\partial x_i \partial x_j}g(\sqrt{1-t}F+\sqrt{t}b) \mathfrak{D}_kF_i+\sum\limits_{i,l=1}^d R^{(k,f^{t,b}_j)}_{il},\,\,1\leq k\leq n,$$
where the remainder terms $R^{(k,f^{t,b}_j)}_{il}$ are bounded by
\begin{equation*}
|R^{(k,f^{t,b}_j)}_{il}|\leq \frac{ \|(f^{t,b}_j)''\|_\infty}{2}[(\mathfrak{d}_kF_i)^2+(\mathfrak{d}_kF_l)^2]\leq \frac{(1-t) \|g'''\|_\infty}{2}[(\mathfrak{d}_kF_i)^2+(\mathfrak{d}_kF_l)^2].
\end{equation*}
We therefore can write
\begin{equation}\label{llmkod5}
E\big[\sum\limits_{j=1}^dF_jf^{t,b}_j(F)\big]=\sqrt{1-t}E\big[\sum\limits_{i,j=1}^d\frac{\partial^2}{\partial x_i\partial x_j}g(\sqrt{1-t}F+\sqrt{t}b)Z^{(\alpha)}_{ij}\big]+R^{t,b}.
\end{equation}
By using the same arguments as in the proof of (\ref{o2md1}), the remainder $R^{t,b}$ satisfies
\begin{equation}\label{ll35do}
|R^{t,b}|\leq (1-t) \|g'''\|_\infty d^2 \sum\limits_{i=1}^d \sum\limits_{k=1}^nE|\mathfrak{D}_kF_i|^3\,\,\forall\,t\in(0,1),b\in \mathbb{R}^d.
\end{equation}
Since $F$ and $Y$ are independent, the relation (\ref{llmkod5}) gives us
\begin{align}
E\big[\sum\limits_{j=1}^d\frac{\partial}{\partial x_j}&g(\sqrt{1-t}F+\sqrt{t}Y)F_j\big]=E\bigg[E\big[\sum\limits_{j=1}^d\frac{\partial}{\partial x_j}g(\sqrt{1-t}F+\sqrt{t}b)F_j\big]|_{b=Y}\bigg]\notag\\
&=E\bigg[E\big[\sum\limits_{j=1}^dF_jf^{t,b}_j(F)\big]|_{b=Y}\bigg]\notag\\
&=\sqrt{1-t}E\big[\sum\limits_{i,j=1}^d\frac{\partial^2}{\partial x_i\partial x_j}g(\sqrt{1-t}F+\sqrt{t}Y)Z^{(\alpha)}_{ij}\big]+E[R^{t,Y}]\label{iim2}.
\end{align}
{\it Step 3.} Inserting (\ref{steinko2}) and (\ref{iim2}) into (\ref{k9}) yields
\begin{equation*}
H'(t)=\frac{1}{2}E\big[\sum\limits_{i,j=1}^d\frac{\partial^2}{\partial x_i\partial x_j}g(\sqrt{1-t}F+\sqrt{t}Y)(C_{ij}-Z^{(\alpha)}_{ij})\big]-\frac{E[R^{t,Y}]}{2\sqrt{1-t}}\,\,\forall\,t\in(0,1).
\end{equation*}
By (\ref{ll35do}), it holds that
\begin{align*}
&|H'(t)|\leq \frac{\|g''\|_\infty}{2}\sum\limits_{i,j=1}^dE|C_{ij}-Z^{(\alpha)}_{ij}|+\frac{\sqrt{1-t} \|g'''\|_\infty d^2}{2}  \sum\limits_{i=1}^d \sum\limits_{k=1}^nE|\mathfrak{D}_kF_i|^3\,\,\forall\,t\in(0,1).
%&\leq \frac{\|g''\|_\infty d}{2}\big(\sum\limits_{i,j=1}^dE|C_{ij}-Z_{ij}|^2\big)^{1/2}+\frac{\sqrt{1-t} \|g'''\|_\infty d^2}{2}  \sum\limits_{i=1}^d \sum\limits_{k=1}^nE|\mathfrak{D}_kF_i|^3\,\,\forall\,t\in(0,1).
\end{align*}
By the definition of $H(t)$ we obtain
\begin{align*}
|E[g(F)]&-E[g(Y)]|=|H(1)-H(0)|=\big|\int_0^1 H'(t)dt\big|\\
&\leq \frac{\|g''\|_\infty }{2}\sum\limits_{i,j=1}^dE|C_{ij}-Z^{(\alpha)}_{ij}|+\frac{\|g'''\|_\infty d^2}{3}  \sum\limits_{i=1}^d \sum\limits_{k=1}^nE|\mathfrak{D}_kF_i|^3.
\end{align*}
So (\ref{9hlf4dj}) is verified. Since $E|C_{ij}-Z^{(\alpha)}_{ij}|\leq |C_{ij}-\sigma_{ij}|+E|\sigma_{ij}-Z^{(\alpha)}_{ij}|\leq |C_{ij}-\sigma_{ij}|+ \sqrt{Var(Z^{(\alpha)}_{ij})},$ we obtain (\ref{9hlw3g3}) from (\ref{9hlf4dj}). This completes the proof.
\end{proof}
Let us end this section with some remarks.
\begin{rem}For one-dimensional nonlinear statistics, we refer the reader to the works made by Chatterjee \cite{Chatterjee2008} and Chen \& R\"ollin \cite{Chen2010}. The Stein's method approximation in the present paper is a natural extension of their approach to the multi-dimensional setting.
\end{rem}
\begin{rem} We have implicitly assumed that the bounds (\ref{9hlw3g}) and (\ref{9hlw3g3}) both involve finite quantities, as otherwise there is nothing to prove. The bounds (\ref{9hlw3g}) and (\ref{9hlw3g3}) only differ by multiplicative constants. Thus Stein's method and Slepian's interpolation method provide us the same criterion for proving the multivariate CLTs.
\end{rem}
\begin{rem}Generally, bounds for the rate of convergence defined via non-smooth test functions are more informative in practice. For instance, such bounds can  be  used  for  the  construction  of  confidence intervals. In this paper, we only discuss  the bounds defined via smooth test functions. However, we note that our bounds can be used to evaluate the bounds for non-smooth test functions. The reader can consult Corollary 7.3 in \cite{Nourdin2010b} and Section 3 in \cite{Reinert2009} for such evaluations.
\end{rem}

\section{Applications}\label{mettt02} In this section, we provide some examples to illustrate the applicability of our abstract results. Even though Theorems \ref{k1opa4} and \ref{ddlm3} are designed to handle very general functions of independent random variables, they prove to be surprisingly simple in studying CLTs for well-known functions such runs and quadratic forms.
\subsection{New normal approximation bounds  for Rademacher functionals}
%\subsection{Rademacher functionals and discrete Malliavin derivative}
%\subsection{Connection with Malliavin-Stein method for Rademacher functionals}
In this subsection, we consider a very special case where $X_1,X_2,...,X_n$ are independent identically distributed Rademacher random variables, i.e. $P(X_i=1)=P(X_i=-1)=\frac{1}{2}.$ The $\mathbb{R}$-valued random variable $U:=U(X_1,X_2,...,X_n)$ is called a Rademacher functional. In the last years, the Malliavin-Stein method has been intensively used to study the normal approximation for Rademacher functionals (see \cite{Krokowski2017} and references therein). Our aim here is to show a connection between our technique with Malliavin-Stein method developed for Rademacher functionals. As a consequence, we obtain new error bounds in the multivariate normal approximation for Rademacher functionals which are stated in terms of Malliavin derivative operator.

Let $F=(F_1,...,F_d)$ be a $\mathbb{R}^d$-valued random vector of centered Rademacher functionals and $Y$ be a centered $d$-dimensional Gaussian vector with covariance matrix $C=(C_{ij})_{d\times d}.$ The first multivariate results were obtained by Krokowski et al. in \cite{Krokowski2016}. Because of certain technical reasons, they have to use the test function of the class $\mathcal{C}^4(\mathbb{R}^d)$ to investigate the rate of convergence. In fact, they define the distance
$$d_4(F,Y):=\sup\limits_{\|g^{(k)}\|_\infty\leq 1,k=1,...,4}|E[g(F)]-E[g(Y)]|,$$
and established the following bound (see Theorem 5.1 of \cite{Krokowski2016})
\begin{align}\label{9km3}
d_4(F,Y)\leq \frac{d}{2}\big(\sum\limits_{i,j=1}^dE|C_{ij}-T_{ij}|^2\big)^{1/2}+\frac{5}{3}\sum\limits_{k=1}^nE\big[\big(\sum\limits_{i=1}^d|D_k F_i|\big)^3\big(\sum\limits_{i=1}^d|D_kL^{-1} F_i|\big)\big],
\end{align}
where $T_{ij}:=\sum\limits_{k=1}^nD_k F_i D_k L^{-1}F_j,$ $D$ denotes the discrete Malliavin derivative operator and $L^{-1}$ is the pseudo-inverse of Ornstein-Uhlenbeck operator. The reader can consult \cite{Privault2008} for more details about Malliavin calculus of Rademacher functionals. We only recall here that, for $U:=U(X_1,X_2,...,X_n),$
$$D_kU:=\frac{U^+_k-U^-_k}{2},\,\,1\leq k\leq n,$$
where $U^+_k=U(X_1,...,X_{k-1},+1,X_{k+1}...,X_n)$ and $U^-_k=U(X_1,...,X_{k-1},-1,X_{k+1}...,X_n).$

\begin{lem}\label{lemma41} The difference operator $\mathfrak{D}$ relates to Malliavin derivative operator $D$ as follows
$$\mathfrak{D}_kU=X_kD_kU,\,\,1\leq k\leq n.$$
\end{lem}
\begin{proof} For $k=1,...,n$ we have
\begin{align*}
\mathfrak{D}_kU&=U-E_k[U]=U^+_k\ind_{\{X_k=1\}}+U^-_k\ind_{\{X_k=-1\}}-\frac{U^+_k+U^-_k}{2}\\
&=\frac{U^+_k-U^-_k}{2}\ind_{\{X_k=1\}}+\frac{U^-_k-U^+_k}{2}\ind_{\{X_k=-1\}}\\
&=\frac{U^+_k-U^-_k}{2}(\ind_{\{X_k=1\}}-\ind_{\{X_k=-1\}})=X_kD_kU.
\end{align*}
This finishes the proof.
\end{proof}
We now apply the results of Section \ref{mettt01} to derive new normal approximation bounds  for Rademacher functionals. Following \cite{Peccati2010}, we consider the distances
$$d_3(F,Y):=\sup\limits_{\|g^{(k)}\|_\infty\leq 1,k=1,...,3}|E[g(F)]-E[g(Y)]|,$$
$$d_2(F,Y):=\sup\limits_{\|g\|_{Lip}\leq 1,M_2(g)\leq 1}|E[g(F)]-E[g(Y)]|.$$

\begin{thm}\label{thm4d1}Suppose that $F=(F_1,...,F_d)$ is a $\mathbb{R}^d$-valued random vector of Rademacher functionals such that $E[F_i]=0$ and $E[F_iF_j]<\infty$ for all $0\leq i,j\leq d.$ Then, for any  $\alpha\in[0,1],$ we have
\begin{equation}\label{rade01}
d_3(F,Y)\leq \frac{d}{2}\big(\sum\limits_{i,j=1}^dE|C_{ij}-T^{(\alpha)}_{ij}|^2\big)^{1/2}+\frac{d^2}{3}\sum\limits_{i=1}^d \sum\limits_{k=1}^nE|D_k F_i|^3.
\end{equation}
Assume in addition that $C$ is positive definite, then
\begin{align}
&d_2(F,Y)\notag\\
&\leq \|C^{-1}\|_{op}\|C\|^{1/2}_{op}\big(\sum\limits_{i,j=1}^dE|C_{ij}-T^{(\alpha)}_{ij}|^2\big)^{1/2}+\frac{\sqrt{2\pi}\|C^{-1}\|^{3/2}_{op}\|C\|_{op}d^2}{4} \sum\limits_{i=1}^d \sum\limits_{k=1}^nE|D_k F_i|^3,\label{rade03}
\end{align}
where
$$T^{(\alpha)}_{ij}:=\sum\limits_{k=1}^nD_k F_i (\alpha E[D_kF_j|\mathcal{F}_{k-1}]+(1-\alpha)E[D_kF_j|\mathcal{G}_{k+1}]),\,1\leq i,j\leq d.$$
\end{thm}
\begin{proof}Recalling the definition of $Z^{(\alpha)}_{ij}$ given in Theorem \ref{k1opa4}, we obtain from Lemma \ref{lemma41} that
\begin{align*}
Z^{(\alpha)}_{ij}%&=\sum\limits_{k=1}^n\mathfrak{D}_k F_i \mathfrak{D}^{(\alpha)}_k F_j\\
&=\sum\limits_{k=1}^nX_kD_k F_i (\alpha E[X_kD_kF_j|\mathcal{F}_k]+(1-\alpha)E[X_kD_kF_j|\mathcal{G}_k])\\
&=\sum\limits_{k=1}^nD_k F_i (\alpha E[D_kF_j|\mathcal{F}_k]+(1-\alpha)E[D_kF_j|\mathcal{G}_k]),\,1\leq i,j\leq d.
\end{align*}
Since $D_kF_j$ is independent of $X_k$ for all $1\leq j\leq d,$ this implies that $E[D_kF_j|\mathcal{F}_k]=E[D_kF_j|\mathcal{F}_{k-1}]$ and $E[D_kF_j|\mathcal{G}_k]=E[D_kF_j|\mathcal{G}_{k+1}].$ Hence, $Z^{(\alpha)}_{ij}=T^{(\alpha)}_{ij}$ for all $1\leq i,j\leq d.$ On the other hand, we have $E|\mathfrak{D}_k F_i|^3=E|D_k F_i|^3$ for all $1\leq k\leq n,1\leq i\leq d.$

Thus, by using the Cauchy-Schwarz inequality, (\ref{rade01}) follows from (\ref{9hlf4dj}). Similarly, (\ref{rade03}) follows from (\ref{jjflme}).

The proof of Theorem is complete.
\end{proof}
\begin{rem} When $\alpha=1,$ the random variable $T^{(1)}_{ii}=\sum\limits_{k=1}^nD_k F_i E[D_kF_i|\mathcal{F}_{k-1}]$ was already used by Privault \& Torrisi to investigate the univariate normal approximation for $F_i.$ Thus our Theorem \ref{thm4d1} can be considered as a multivariate extension of Theorem 3.2 in \cite{Privault2015}.
\end{rem}
\begin{rem} The random variables $T_{ij}$ in (\ref{9km3}) are defined via the Ornstein-Uhlenbeck operator. Meanwhile, our random variables $T^{(\alpha)}_{ij}$ require the computation of conditional expectations. Hence, the bound (\ref{9km3}) and our bounds (\ref{rade01}), (\ref{rade03}) provide different ways to verify the multivariate CLTs for Rademacher functionals.
\end{rem}

\subsection{Runs}% and scans
Let $m_1,...,m_d$ be positive integer numbers such that $m_1\leq ...\leq m_d$ and $X_1,...,X_{n+m_d-1}$ be independent $\mathbb{R}$-valued random variables with means $\mu_i=E[X_i]$ and finite fourth moments. For each $j=1,...,d,$ we consider the $m_j$-run $F^{(m_j)}$ defined by
$$F^{(m_j)}:=\sum\limits_{i=1}^n a^{(m_j)}_{i,...,i+m_j-1}(X_i...X_{i+m_j-1}-\mu_i...\mu_{i+m_j-1}),$$%=a_{1,2}X_1X_2+...+a_{n-1,n}X_{n-1}X_n$$
where $ a^{(m_j)}_{i,...,i+m_j-1},1\leq i\leq n,1\leq j\leq d$ are real numbers. The reader can consult the monograph \cite{Balakrishnan2002} for more details about the runs. In this subsection, we investigate the multivariate normal approximation for the vector
$$F:=(F^{(m_1)},...,F^{(m_d)}).$$
\begin{thm}\label{oolmd5}Let $Y$ be a centered $d$-dimensional Gaussian vector with the same covariance matrix as that of $F.$ For any $g\in \mathcal{C}^3(\mathbb{R}^d)$ with $\|g''\|_\infty+\|g'''\|_\infty<\infty,$ we have
\begin{multline}\label{komvv4s}
|E[g(F)]-E[g(Y)]|\leq \sqrt{2}\|g''\|_\infty d\sum\limits_{i=1}^d m_i^3 \sqrt{y^{m_i-1}_{1}y_{2}\sum\limits_{k=1}^n |a^{(m_i)}_{k,...,k+m_i-1}|^4}\\+\frac{\|g'''\|_\infty d^2}{3}\sum\limits_{i=1}^d m_i^3x^{m_i-1}_{1}x_{2}\sum\limits_{k=1}^n |a^{(m_i)}_{k,...,k+m_i-1}|^3,
\end{multline}
where $x_{1}:=\max\limits_{1\leq i\leq n+m_d-1}E|X_i|^3,$ $x_{2}:=\max\limits_{1\leq i\leq n+m_d-1} E|X_i-\mu_i|^3,$ $y_{1}:=\max\limits_{1\leq i\leq n+m_d-1}E|X_i|^4,$ $y_{2}:=\max\limits_{1\leq i\leq n+m_d-1} E|X_i-\mu_i|^4.$
\end{thm}
\begin{proof}
 For any $g\in \mathcal{C}^3(\mathbb{R}^d)$ with $\|g''\|_\infty+\|g'''\|_\infty<\infty,$ Theorem \ref{ddlm3} with $\alpha=1$ gives us
\begin{equation}\label{komvv4}
|E[g(F)]-E[g(Y)]|\leq \frac{\|g''\|_\infty}{2}\sum\limits_{i,j=1}^d \sqrt{Var(Z_{ij})}+\frac{\|g'''\|_\infty d^2}{3}\sum\limits_{j=1}^d \sum\limits_{k=1}^{n+m_j-1}E|\mathfrak{D}_k F^{(m_j)}|^3,
\end{equation}
%$$Z_{ij}:=\sum\limits_{k=1}^n\mathfrak{D}_k F^{(m_i)}E[\mathfrak{D}_k F^{(m_j)}|\mathcal{F}_k],\,1\leq i,j\leq d.$$
where $Z_{ij}:=\sum\limits_{k=1}^{n+(m_i\wedge m_j)-1}\mathfrak{D}_k F^{(m_i)} E[\mathfrak{D}_k F^{(m_j)}|\mathcal{F}_k],\,1\leq i,j\leq d.$

%$B_3:=\frac{\|g''\|_\infty}{2} ,B_4:=\frac{\|g'''\|_\infty d^2}{3}$

We put $X_i=X_1$ for all $i\geq n+m_d$ and use the convention $a^{(m_j)}_{i,...,i+m_j-1}=0$ if $i\leq 0$ or $i>n.$ Then, we have
$$\mathfrak{D}_kF^{(m_j)}=a^{(m_j)}_{k-m_j+1,...,k}X_{k-m_j+1}...X_{k-1}(X_k-\mu_k)+...+a^{(m_j)}_{k,...,k+m_j-1}(X_k-\mu_k)X_{k+1}...X_{k+m_j-1}$$
for all $k\geq 1.$ Hence, we can deduce
\begin{align*}
E|\mathfrak{D}_kF^{(m_j)}|^3&\leq m_j^2 \big(E|a^{(m_j)}_{k-m_j+1,...,k}X_{k-m_j+1}...X_{k-1}(X_k-\mu_k)|^3\\
&\hspace{1.5cm}+...+E|a^{(m_j)}_{k,...,k+m_j-1}(X_k-\mu_k)X_{k+1}...X_{k+m_j-1}|^3\big)\\
&\leq  m_j^2 x^{m_j-1}_{1}x_{2}\left(|a^{(m_j)}_{k-m_j+1,...,k}|^3+...+|a^{(m_j)}_{k,...,k+m_j-1}|^3\right),\,\,k\geq 1
\end{align*}
and
\begin{equation}\label{kortm4a}
\sum\limits_{k=1}^{n+m_j-1}E|\mathfrak{D}_kF^{(m_j)}|^3\leq  m_j^3 x^{m_j-1}_{1}x_{2}\sum\limits_{k=1}^n |a^{(m_j)}_{k,...,k+m_j-1}|^3.
\end{equation}
Similarly, we also have
\begin{equation}
\sum\limits_{k=1}^{n+m_j-1}E|\mathfrak{D}_kF^{(m_j)}|^4\leq  m_j^4 y^{m_j-1}_{1}y_{2}\sum\limits_{k=1}^n |a^{(m_j)}_{k,...,k+m_j-1}|^4.
\end{equation}
We write $$Z_{ij}=\sum\limits_{k=1}^{n+(m_i\wedge m_j)-1}Z^{(k)}_{ij},\,\,1\leq i,j\leq d,$$ where
$Z^{(k)}_{ij}:=\mathfrak{D}_k F^{(m_i)} E[\mathfrak{D}_k F^{(m_j)}|\mathcal{F}_k].$ Note that
$$E|Z^{(k)}_{ij}|^2=E|\mathfrak{D}_k F^{(m_i)} E[\mathfrak{D}_k F^{(m_j)}|\mathcal{F}_k]|^2\leq \frac{1}{2}\left(E|\mathfrak{D}_k F^{(m_i)}|^4+E|\mathfrak{D}_k F^{(m_j)}|^4\right).$$
%$$E|\mathfrak{D}_iF^{(m_j)}|^3\leq 8 m_j^2x^{m_j}_0\left(|a^{(m_j)}_{i-m_j+1,...,i}|^3+...+|a^{(m_j)}_{i,...,i+m_j-1}|^3\right)$$
%\begin{equation}
%\sum\limits_{i=1}^{n+m_j-1}E|\mathfrak{D}_iF^{(m_j)}|^3\leq 8 m_j^3x^{m_j}_0\sum\limits_{i=1}^n |a^{(m_j)}_{i,...,i+m_j-1}|^3
%\end{equation}
%\begin{equation}
%\sum\limits_{i=1}^{n+m_j-1}E|\mathfrak{D}_iF^{(m_j)}|^4\leq 16 m_j^4y^{m_j}_0\sum\limits_{i=1}^n |a^{(m_j)}_{i,...,i+m_j-1}|^4
%\end{equation}
Using the convention $Z^{(k)}_{ij}=0$ if $k\leq 0,$ we have
$$\mathfrak{D}_lZ_{ij}=\mathfrak{D}_lZ^{(l-(m_i\wedge m_j)+1)}_{ij}+...+\mathfrak{D}_lZ^{(l)}_{ij}+...+\mathfrak{D}_lZ^{(l+(m_i\wedge m_j)-1)}_{ij},\,\,1\leq l\leq n+(m_i\wedge m_j)-1$$
because $\mathfrak{D}_lZ^{(k)}_{ij}=0$ if $Z^{(k)}_{ij}$ does not depend on $X_l.$  By using the Cauchy-Schwarz inequality and then Proposition \ref{mcsk4}, $(iv)$ we obtain
\begin{align*}
&E|\mathfrak{D}_lZ_{ij}|^2\\
&\leq (2(m_i\wedge m_j)-1)\left(E|\mathfrak{D}_lZ^{(l-(m_i\wedge m_j)+1)}_{ij}|^2+...+E|\mathfrak{D}_lZ^{(l)}_{ij}|^2+...+E|\mathfrak{D}_lZ^{(l+(m_i\wedge m_j)-1)}_{ij}|^2\right)\\
&\leq (2(m_i\wedge m_j)-1)\left(E|Z^{(l-(m_i\wedge m_j)+1)}_{ij}|^2+...+E|Z^{(l)}_{ij}|^2+...+E|Z^{(l+(m_i\wedge m_j)-1)}_{ij}|^2\right)
\end{align*}
for $1\leq l\leq n+(m_i\wedge m_j)-1.$ We now use Efron-Stein inequality (\ref{sopq1}) to estimate $Var(Z_{ij}).$ We have
\begin{align*}Var(Z_{ij})&\leq \sum\limits_{l=1}^{n+(m_i\wedge m_j)-1}E|\mathfrak{D}_lZ_{ij}|^2\\
&\leq (2(m_i\wedge m_j)-1)^2 \sum\limits_{l=1}^{n+(m_i\wedge m_j)-1}E|Z^{(l)}_{ij}|^2\\
&\leq \frac{1}{2}(2(m_i\wedge m_j)-1)^2\left(\sum\limits_{l=1}^{n+m_i-1}E|\mathfrak{D}_lF^{(m_i)}|^4+\sum\limits_{l=1}^{n+m_j-1}E|\mathfrak{D}_lF^{(m_j)}|^4\right)\\
&\leq 2\left(m_i^2\sum\limits_{l=1}^{n+m_i-1}E|\mathfrak{D}_lF^{(m_i)}|^4+m_j^2\sum\limits_{l=1}^{n+m_j-1}E|\mathfrak{D}_lF^{(m_j)}|^4\right)\\
&\leq 2\left(m_i^6y^{m_i-1}_{1}y_{2}\sum\limits_{k=1}^n |a^{(m_i)}_{k,...,k+m_i-1}|^4+m_j^6 y^{m_j-1}_{1}y_{2}\sum\limits_{k=1}^n |a^{(m_j)}_{k,...,k+m_j-1}|^4\right).
\end{align*}
We therefore obtain
$$\sqrt{Var(Z_{ij})}\leq \sqrt{2}\left(m_i^3\sqrt{y^{m_i-1}_{1}y_{2}\sum\limits_{k=1}^n |a^{(m_i)}_{k,...,k+m_i-1}|^4}+m_j^3 \sqrt{y^{m_j-1}_{1}y_{2}\sum\limits_{k=1}^n |a^{(m_j)}_{k,...,k+m_j-1}|^4}\right)$$
and
\begin{equation}\label{kortm4}
\sum\limits_{i,j=1}^d \sqrt{Var(Z_{ij})}\leq 2\sqrt{2}d\sum\limits_{i=1}^d m_i^3\sqrt{y^{m_i-1}_{1}y_{2}\sum\limits_{k=1}^n |a^{(m_i)}_{k,...,k+m_i-1}|^4}.
\end{equation}
So we can get (\ref{komvv4s}) by inserting (\ref{kortm4a}) and (\ref{kortm4}) into (\ref{komvv4}). This completes the proof.
%\begin{align*}
%\sum\limits_{i,j=1}^d Var(Z_{ij})&\leq 4d\sum\limits_{i=1}^d m_i^6y^{m_i-1}_{1}y_{2}\sum\limits_{k=1}^n |a^{(m_i)}_{k,...,k+m_i-1}|^4
%&\leq 4(2m_d-1)^2d\sum\limits_{i=1}^dm_i^4y^{m_i}_0\sum\limits_{l=1}^n |a^{(m_i)}_{l,...,l+m_i-1}|^4
%\end{align*}
\end{proof}

\begin{rem}A very special case of Theorem \ref{oolmd5} has been discussed in \cite{Reinert2009}: Let $ X_i's$ be independent random variables with distribution $\text{Bernoulli}(p),0<p<1.$ We define $j$-run $W_j$ by
$$W_j=\sum\limits_{i=1}^n \frac{1}{\sqrt{np^j(1-p)}}(X_i...X_{i+j-1}-p^j),\,\,1\leq j\leq d$$
and consider the vector $W=(W_1,...,W_d).$ Let $Y$ be a centered $d$-dimensional Gaussian vector with covariance matrix $(\sigma_{ij})_{d\times d}$ defined by
$$\sigma_{ij}:=E[W_iW_j]=p^{|i-j|/2}\sum\limits_{k=0}^{i\wedge j-1}(|i-j|+1+2k)p^k.$$
Theorem 4.1 in \cite{Reinert2009} provides the following rate of convergence
\begin{equation}\label{klsm8}
|E[g(W)]-E[g(Y)]|\leq \frac{416d^{7/2}\|g''\|_\infty+960d^5\|g'''\|_\infty}{p^{d/2}(1-p)^{3/2}\sqrt{n}}.
\end{equation}
Let us now apply Theorem \ref{oolmd5} to $F=W.$ We have $m_i=i$ and $a^{(m_i)}_{k,...,k+m_i-1}=\frac{1}{\sqrt{np^i(1-p)}}$ for each $1\leq i\leq d.$ We also have $x_{1}=y_{1}=p,$ $x_{2}=(1-p)^3p+p^3(1-p)\leq p(1-p)$ and $y_{2}=(1-p)^4p+p^4(1-p)\leq p(1-p).$ Hence, it holds that
\begin{align*}
\sum\limits_{i=1}^d m_i^3x^{m_i-1}_{1}x_{2}\sum\limits_{k=1}^n |a^{(m_i)}_{k,...,k+m_i-1}|^3&\leq \sum\limits_{i=1}^d i^3p^i(1-p)\frac{n}{\sqrt{n^3p^{3i}(1-p)^3}}\\
&\leq \frac{d^3}{\sqrt{n(1-p)}}\sum\limits_{i=1}^d \frac{1}{p^{i/2}}\leq\frac{d^3}{\sqrt{n(1-p)}} \frac{2}{p^{d/2}(1-p)}
\end{align*}
and
$$\sum\limits_{i=1}^d m_i^3 \sqrt{y^{m_i-1}_{1}y_{2}\sum\limits_{k=1}^n |a^{(m_i)}_{k,...,k+m_i-1}|^4}\leq \frac{d^3}{\sqrt{n(1-p)}}\sum\limits_{i=1}^d \frac{1}{p^{i/2}}\leq\frac{d^3}{\sqrt{n(1-p)}} \frac{2}{p^{d/2}(1-p)}.$$
Combining the above estimates with (\ref{komvv4s}) we obtain
$$|E[g(W)]-E[g(Y)]|\leq \frac{2\sqrt{2}d^{4}\|g''\|_\infty+\frac{2}{3}d^5\|g'''\|_\infty}{p^{d/2}(1-p)^{3/2}\sqrt{n}},$$
which is better than (\ref{klsm8}) when, for example, the dimension $d$ such that $2\sqrt{2}d^{4}\leq 416d^{7/2}$ or $d\leq 21632.$
\end{rem}

\subsection{Multivariate CLT for quadratic forms}
Suppose $X_1,...,X_{n}$ are independent $\mathbb{R}$-valued random variables with zero means, unit variances and finite fourth moments. Let $A=(a^{(n)}_{uv})_{u,v=1}^n$  be a real symmetric matrix with vanishing diagonal, i.e. $a^{(n)}_{uv}=a^{(n)}_{vu}$ and $a^{(n)}_{uu}=0.$ The central limit theorem (CLT) for the quadratic form
$$W_n=\sum\limits_{1\leq u\leq v\leq n} a^{(n)}_{uv}X_uX_v$$
has been extensively discussed in the literature. The best known result given by de Jong \cite{deJong1987} says that the $\sigma_n^{-1}W_n$ converges to a standard normal random variable in distribution if
$$\text{$\sigma_n^{-4}\mathrm{Tr}(A^4)\to 0$ and $\sigma_n^{-2}\max\limits_{1\leq u\leq n}\sum\limits_{v=1}^n(a^{(n)}_{uv})^2\to 0$},$$
where $\sigma^2_n:=Var(W_n)=\sum\limits_{1\leq u\leq v\leq n} (a^{(n)}_{uv})^2.$ We recall that $\mathrm{Tr}(A^4)=\sum\limits_{u,v=1}^n \big(\sum\limits_{k=1}^na^{(n)}_{ku} a^{(n)}_{kv}\big)^2$ and the first condition is equivalent to $\sigma_n^{-4}E[W^4_n]\to 3.$

 In this section, we generalize this classical result to multi-dimensional setting. Let $A_i=(a^{(ni)}_{uv})_{u,v=1}^n$  be real symmetric matrices with vanishing diagonal, we define the quadratic forms
%$F_i=\sum\limits_{j=1}^n a_{ij}X_j$
%$F_i=\sum\limits_{j=1}^n a^{(i)}_{j}X_j$
$$F^{(n)}_i:=\sum\limits_{1\leq u\leq v\leq n} a^{(ni)}_{uv}X_uX_v,\,\,1\leq i\leq d$$
and consider the $\mathbb{R}^d$-valued vector
$$F^{(n)}:=(F^{(n)}_1,...,F^{(n)}_{d}).$$
It is interesting to mention that the condition (\ref{condi02}) with $i=j$ is equivalent to the fourth moment condition required in Theorem 1.7 of \cite{Dobler2017b}, provided that $C_{ii}=1$ for $1\leq i\leq d.$
\begin{thm}\label{jong} Let $Y$ be a centered $d$-dimensional Gaussian vector with covariance matrix $C=(C_{ij})_{d\times d}$ (not
necessarily positive definite). Suppose that
\begin{equation}\label{condi01}
\lim\limits_{n\to\infty}E[F^{(n)}_iF^{(n)}_j]=\lim\limits_{n\to\infty}\sum\limits_{1\leq u\leq v\leq n}  a^{(ni)}_{uv} a^{(nj)}_{uv}=C_{ij},\,\,1\leq i,j\leq d,
\end{equation}
\begin{equation}\label{condi02}
\lim\limits_{n\to\infty}\sum\limits_{u,v=1}^n \big(\sum\limits_{k=1}^na^{(ni)}_{ku} a^{(nj)}_{kv}\big)^2=0,\,\,1\leq i,j\leq d,
\end{equation}
\begin{equation}\label{condi03}
\lim\limits_{n\to\infty}\max\limits_{1\leq u\leq n}\sum\limits_{v=1}^n(a^{(ni)}_{uv})^2=0,\,\,1\leq i\leq d.
\end{equation}
Then, $F^{(n)}$ converges to $Y$ in distribution as $n\to\infty.$ Moreover, we have the following bound for the rate of convergence
\begin{align}
&|E[g(F^{(n)})]-E[g(Y)]|\leq \frac{\|g''\|_\infty}{2}\sum\limits_{i,j=1}^d|C_{ij}-E[F^{(n)}_iF^{(n)}_j]|\notag\\
&+\frac{\|g''\|_\infty}{2^{3/2}}\sum\limits_{i,j=1}^d\sqrt{ \max\{2,\max\limits_{1\leq u\leq n}Var(X^2_u)\}\sum\limits_{u,v=1}^n \big(\sum\limits_{k=1}^na^{(ni)}_{ku} a^{(nj)}_{kv}\big)^2}\notag\\
&+\frac{\|g''\|_\infty}{2^{3/2}}\sum\limits_{i,j=1}^d \sqrt{8\max\limits_{1\leq v\leq n}Var(X^2_v)\max\limits_{1\leq v\leq n}E|X_v|^4\sum\limits_{k=1}^n\big(\sum\limits_{v=1}^n(a^{(ni)}_{kv})^2\big)\big(\sum\limits_{v=1}^n(a^{(nj)}_{kv})^2\big)}\notag\\
&+\frac{2^{3/2}\max\limits_{1\leq v\leq n}E|X_v|^4\|g'''\|_\infty d^2}{3}\sum\limits_{i=1}^d \sum\limits_{k=1}^n\big(\sum\limits_{u=1}^n(a^{(ni)}_{ku})^2\big)^{3/2},\label{jjdl5}
\end{align}
where $g\in \mathcal{C}^3(\mathbb{R}^d)$ with $\|g''\|_\infty+\|g'''\|_\infty<\infty.$
\end{thm}
\begin{proof}We first use Theorem \ref{ddlm3} with $\alpha=\frac{1}{2}$ to verify the bound (\ref{jjdl5}). For any $g\in \mathcal{C}^3(\mathbb{R}^d)$ with $\|g''\|_\infty+\|g'''\|_\infty<\infty,$ we have
\begin{multline}\label{jjdl55}
|E[g(F^{(n)})]-E[g(Y)]|\leq \frac{\|g''\|_\infty}{2}\sum\limits_{i,j=1}^d|C_{ij}-E[F^{(n)}_iF^{(n)}_j]|\\
+\frac{\|g''\|_\infty}{2}\sum\limits_{i,j=1}^d \sqrt{Var(Z^{(\frac{1}{2})}_{ij})}
+\frac{\|g'''\|_\infty d^2}{3}\sum\limits_{i=1}^d \sum\limits_{k=1}^nE|\mathfrak{D}_k F^{(n)}_i|^3.
\end{multline}
For each $k=1,...,n$ we have $\mathfrak{D}_kF^{(n)}_i=X_k\sum\limits_{v=1}^n a^{(ni)}_{kv}X_v,$
$$E[\mathfrak{D}_kF^{(n)}_i|\mathcal{F}_k]=X_k\sum\limits_{v=1}^k a^{(ni)}_{kv}X_v\,\,\,\text{and}\,\,\,E[\mathfrak{D}_kF^{(n)}_i|\mathcal{G}_k]=X_k\sum\limits_{v=k}^n a^{(ni)}_{kv}X_v.$$
Then we obtain $Z^{(\frac{1}{2})}_{ij}=\frac{1}{2}(Z^\ast_{ij}+Z^\star_{ij}),\,\,1\leq i,j\leq d,$ where
$$Z^\ast_{ij}:=\sum\limits_{k=1}^n(X^2_k-1)\sum\limits_{v=1}^n a^{(i)}_{kv}X_v\sum\limits_{v=1}^n a^{(j)}_{kv}X_v,\,\,Z^\star_{ij}:=\sum\limits_{k=1}^n\sum\limits_{v=1}^n a^{(i)}_{kv}X_v\sum\limits_{v=1}^n a^{(j)}_{kv}X_v.$$
%\begin{align*}
%Z_{ij}&=\sum\limits_{k=1}^n \mathfrak{D}_kG_i \mathfrak{D}^{(\frac{1}{2})}_kG_j =\frac{1}{2}\sum\limits_{k=1}^nX^2_k\sum\limits_{v=1}^n a^{(ni)}_{kv}X_v\sum\limits_{v=1}^n a^{(nj)}_{kv}X_v\\
%&=\frac{1}{2}\sum\limits_{k=1}^n(X^2_k-1)\sum\limits_{v=1}^n a^{(i)}_{kv}X_v\sum\limits_{v=1}^n a^{(j)}_{kv}X_v+\frac{1}{2}\sum\limits_{k=1}^n\sum\limits_{v=1}^n a^{(i)}_{kv}X_v\sum\limits_{v=1}^n a^{(j)}_{kv}X_v\\
%&=:\frac{1}{2}(Z^\ast_{ij}+Z^\star_{ij}),\,\,1\leq i,j\leq d.
%\end{align*}
Hence,
$$Var(Z^{(\frac{1}{2})}_{ij})\leq \frac{1}{2}\left(Var(Z^\ast_{ij})+Var(Z^\star_{ij})\right),\,\,1\leq i,j\leq d.$$
To estimate $Var(Z^\ast_{ij}),$ we put
$$Z^{\ast (k)}_{ij}:=(X^2_k-1)\sum\limits_{v=1}^n a^{(ni)}_{kv}X_v\sum\limits_{v=1}^n a^{(nj)}_{kv}X_v,\,\,1\leq k\leq n.$$
We have $\mathfrak{D}_l Z^{\ast (k)}_{ij}=0$ if $l=k$ and for $l\neq k,$
$$\mathfrak{D}_l Z^{\ast (k)}_{ij}=(X^2_k-1)\bigg(a^{(ni)}_{kl}a^{(nj)}_{kl}(X^2_l-1)+a^{(ni)}_{kl}X_l\sum\limits_{v=1,v\neq l}^n a^{(nj)}_{kv}X_v+a^{(nj)}_{kl}X_l\sum\limits_{v=1,v\neq l}^n a^{(ni)}_{kv}X_v\bigg).$$
Hence,
$$\mathfrak{D}_l Z^\ast_{ij}=\sum\limits_{k=1,k\neq l}^n \mathfrak{D}_l Z^{\ast (k)}_{ij},\,\,1\leq l\leq n$$
and by Efron-Stein inequality (\ref{sopq1}) we obtain
\begin{align*}Var(Z^\ast_{ij})&\leq \sum\limits_{l=1}^n E|\mathfrak{D}_l Z^\ast_{ij}|^2= \sum\limits_{l=1}^n E\big|\sum\limits_{k=1,k\neq l}^n \mathfrak{D}_l Z^{\ast (k)}_{ij}\big|^2\\
&=\sum\limits_{l=1}^n \sum\limits_{k=1,k\neq l}^n E|\mathfrak{D}_l Z^{\ast (k)}_{ij}|^2+\sum\limits_{l=1}^n \sum\limits_{k\neq k';k,k'\neq l}E[ \mathfrak{D}_l Z^{\ast (k)}_{ij} \mathfrak{D}_l Z^{\ast (k')}_{ij}].
\end{align*}
%We observe that $E[ \mathfrak{D}_l Z^{\ast (k)}_{ij} \mathfrak{D}_l Z^{\ast (k')}_{ij}]=0$ for $k\neq k'.$ As a consequence,
%\begin{align*}Var(Z^\ast_{ij})\leq \sum\limits_{l=1}^n \sum\limits_{k=1,k\neq l}^n E|\mathfrak{D}_l Z^{\ast (k)}_{ij}|^2= \sum\limits_{l=1}^n \sum\limits_{k=1}^n E|\mathfrak{D}_l Z^{\ast (k)}_{ij}|^2.
%\end{align*}
By the independence and the elementary inequality $|a+b|^2\leq 2(a^2+b^2)$ we deduce
\begin{align*}
&E|\mathfrak{D}_l Z^{\ast (k)}_{ij}|^2\\
&=Var(X^2_k)\bigg((a^{(ni)}_{kl})^2(a^{(nj)}_{kl})^2Var(X^2_l)+E|a^{(ni)}_{kl}X_l\sum\limits_{v=1,v\neq l}^n a^{(nj)}_{kv}X_v+a^{(nj)}_{kl}X_l\sum\limits_{v=1,v\neq l}^n a^{(ni)}_{kv}X_v|^2\bigg)\\
&\leq Var(X^2_k)\bigg((a^{(ni)}_{kl})^2(a^{(nj)}_{kl})^2Var(X^2_l)+2(a^{(nj)}_{kl})^2\sum\limits_{v=1,v\neq l}^n (a^{(ni)}_{kv})^2+2 (a^{(ni)}_{kl})^2\sum\limits_{v=1,v\neq l}^k (a^{(nj)}_{kv})^2\bigg)\\
&\leq2\max\limits_{1\leq v\leq n}Var(X^2_v)\max\limits_{1\leq v\leq n}E|X_v|^4\bigg((a^{(nj)}_{kl})^2\sum\limits_{v=1}^n (a^{(ni)}_{kv})^2+ (a^{(ni)}_{kl})^2\sum\limits_{v=1}^k (a^{(nj)}_{kv})^2\bigg)
\end{align*}
and hence,
\begin{equation}\label{mmesk}
\sum\limits_{l=1}^n \sum\limits_{k=1,k\neq l}^n E|\mathfrak{D}_l Z^{\ast (k)}_{ij}|^2\leq 4\max\limits_{1\leq v\leq n}Var(X^2_v)\max\limits_{1\leq v\leq n}E|X_v|^4\sum\limits_{k=1}^n\big(\sum\limits_{v=1}^n(a^{(ni)}_{kv})^2\big)\big(\sum\limits_{v=1}^n(a^{(nj)}_{kv})^2\big).
\end{equation}
From the decomposition
$$\mathfrak{D}_l Z^{\ast (k)}_{ij}=(X^2_k-1)\bigg(A^{(k)}(\neq k,k')+a^{(ni)}_{kl}X_la^{(nj)}_{kk'}X_{k'}+a^{(nj)}_{kl}X_la^{(ni)}_{kk'}X_{k'}\bigg),$$
where the term $A^{(k)}(\neq k,k')$ does not depend on $X_k$ and $X_{k'},$ we obtain
\begin{align*}
&E[ \mathfrak{D}_l Z^{\ast (k)}_{ij} \mathfrak{D}_l Z^{\ast (k')}_{ij}]\\
&=E[(X^2_k-1)(X^2_{k'}-1)(a^{(ni)}_{kl}a^{(nj)}_{kk'}+a^{(nj)}_{kl}a^{(ni)}_{kk'})
(a^{(ni)}_{k'l}a^{(nj)}_{k'k}+a^{(nj)}_{k'l}a^{(ni)}_{k'k})X^2_lX_{k}X_{k'}]\\
&=E[(X^2_k-1)X_{k}]E[(X^2_{k'}-1)X_{k'}](a^{(ni)}_{kl}a^{(nj)}_{kk'}+a^{(nj)}_{kl}a^{(ni)}_{kk'})
(a^{(ni)}_{k'l}a^{(nj)}_{k'k}+a^{(nj)}_{k'l}a^{(ni)}_{k'k})\\
&\leq \max\limits_{1\leq v\leq n}Var(X^2_v)(a^{(ni)}_{kl}a^{(nj)}_{kk'}+a^{(nj)}_{kl}a^{(ni)}_{kk'})
(a^{(ni)}_{k'l}a^{(nj)}_{k'k}+a^{(nj)}_{k'l}a^{(ni)}_{k'k})\\
&\leq \max\limits_{1\leq v\leq n}Var(X^2_v)\big((a^{(ni)}_{kl})^2(a^{(nj)}_{kk'})^2+(a^{(nj)}_{kl})^2(a^{(ni)}_{kk'})^2
+(a^{(ni)}_{k'l})^2(a^{(nj)}_{k'k})^2+(a^{(nj)}_{k'l})^2(a^{(ni)}_{k'k})^2\big),
\end{align*}
which implies that
$$\sum\limits_{l=1}^n \sum\limits_{k\neq k';k,k'\neq l}E[ \mathfrak{D}_l Z^{\ast (k)}_{ij} \mathfrak{D}_l Z^{\ast (k')}_{ij}]\leq 4\max\limits_{1\leq v\leq n}Var(X^2_v)\sum\limits_{k=1}^n\big(\sum\limits_{v=1}^n(a^{(ni)}_{kv})^2\big)\big(\sum\limits_{v=1}^n(a^{(nj)}_{kv})^2\big).$$
This, together with (\ref{mmesk}), yields
\begin{equation}\label{do02a1}
Var(Z^\ast_{ij})\leq 8\max\limits_{1\leq v\leq n}Var(X^2_v)\max\limits_{1\leq v\leq n}E|X_v|^4\sum\limits_{k=1}^n\big(\sum\limits_{v=1}^n(a^{(ni)}_{kv})^2\big)\big(\sum\limits_{v=1}^n(a^{(nj)}_{kv})^2\big).
\end{equation}
On the other hand, it is easy to estimate $Var(Z^\star_{ij}).$ Indeed, we have
$$Z^\star_{ij}=\sum\limits_{k=1}^n\sum\limits_{u,v=1}^n a^{(i)}_{ku} a^{(j)}_{kv}X_uX_v=\sum\limits_{u,v=1}^n\big(\sum\limits_{k=1}^na^{(i)}_{ku} a^{(j)}_{kv}\big)X_uX_v.$$
%=\frac{1}{2}\sum\limits_{u,v=1}^nA_{uv}X_uX_v$A_{uv}=\sum\limits_{k=1}^na^{(i)}_{ku} a^{(j)}_{kv}+\sum\limits_{k=1}^na^{(j)}_{ku} a^{(i)}_{kv}$
%$$Var(Z^\star_{ij})=\frac{2}{4}\sum\limits_{u,v=1:u\neq v}^nA^2_{uv}+\frac{1}{4}\sum\limits_{u=1}^nVar(X^2_u)A^2_{uu}\leq \frac{\max\{2,\max\limits_{1\leq u\leq n}Var(X^2_u)\}}{4} \sum\limits_{u,v=1}^nA^2_{uv}$$
Thus $Z^\star_{ij}$ is a quadratic form with nonvanishing diagonal and hence,
\begin{equation}\label{do02a2}
Var(Z^\star_{ij})\leq \max\{2,\max\limits_{1\leq u\leq n}Var(X^2_u)\}\sum\limits_{u,v=1}^n \big(\sum\limits_{k=1}^na^{(ni)}_{ku} a^{(nj)}_{kv}\big)^2.
\end{equation}
%$$\sum\limits_{u,v=1}^n \big(\sum\limits_{k=1}^na^{(i)}_{ku} a^{(j)}_{kv}\big)^2+E|X^2_k-1|^2\sum\limits_{k=1}^n\big(\sum\limits_{u=1}^n(a^{(ni)}_{ku})^2\big)\big(\sum\limits_{u=1}^n(a^{(nj)}_{ku})^2\big)$$
It only remains to estimate $E|\mathfrak{D}_kF^{(n)}_i|^3.$ We use Theorem 2.1 in \cite{Rio2009} to get
\begin{align}
E|\mathfrak{D}_kF^{(n)}_i|^3=E|X_k|^3E|\sum\limits_{v=1}^n a^{(ni)}_{kv}X_v|^3&\leq 2^{3/2}\max\limits_{1\leq v\leq n}(E|X_v|^3)^2\big(\sum\limits_{u=1}^n(a^{(ni)}_{ku})^2\big)^{3/2}\notag\\
&\leq 2^{3/2}\max\limits_{1\leq v\leq n}E|X_v|^4\big(\sum\limits_{u=1}^n(a^{(ni)}_{ku})^2\big)^{3/2}.\label{do02a3}
\end{align}
Recalling $Var(Z^{(\frac{1}{2})}_{ij})\leq \frac{1}{2}\left(Var(Z^\ast_{ij})+Var(Z^\star_{ij})\right),$ we obtain (\ref{jjdl5}) by inserting (\ref{do02a1}), (\ref{do02a2}) and (\ref{do02a3}) into (\ref{jjdl55}).

To prove the convergence of $F^{(n)}$ to $Y$ in distribution, we need to show that
$$|E[g(F^{(n)})]-E[g(Y)]|\to 0\,\,\,\text{as}\,\,\,n\to \infty.$$
The conditions (\ref{condi01}) and (\ref{condi02}) imply that the first two terms in the right hand side of (\ref{jjdl5}) converge to zero, respectively. Moreover, we have
\begin{align*}
\sum\limits_{k=1}^n\big(\sum\limits_{v=1}^n(a^{(ni)}_{kv})^2\big)\big(\sum\limits_{v=1}^n&(a^{(nj)}_{kv})^2\big)
\leq\frac{1}{2}\sum\limits_{k=1}^n\big(\sum\limits_{v=1}^n(a^{(ni)}_{kv})^2\big)^2+
\frac{1}{2}\sum\limits_{k=1}^n\big(\sum\limits_{v=1}^n(a^{(ni)}_{kv})^2\big)^2\\
&\leq E|F^{(n)}_i|^2\left[\max\limits_{1\leq u\leq n}\sum\limits_{v=1}^n(a^{(ni)}_{uv})^2\right]+E|F^{(n)}_j|^2\left[\max\limits_{1\leq u\leq n}\sum\limits_{v=1}^n(a^{(nj)}_{uv})^2\right],
\end{align*}
$$\sum\limits_{u=1}^n\big(\sum\limits_{k=1}^n(a^{(i)}_{ku})^2\big)^{3/2}\leq E|F^{(n)}_i|^2\left[\max\limits_{1\leq u\leq n}\sum\limits_{v=1}^n(a^{(ni)}_{uv})^2\right]^{1/2}$$
and $E|F^{(n)}_i|^2\leq C_{ii}+1, E|F^{(n)}_j|^2\leq C_{jj}+1$ for $n$ large sufficiently. Hence, the condition (\ref{condi03}) ensures that the last two terms in the right hand side of (\ref{jjdl5}) also converge to zero as $n\to\infty.$ This completes the proof.
\end{proof}

\noindent {\bf Acknowledgments.}  The author would like to thank the anonymous referees for their valuable comments for improving the paper.

\end{document}